\documentclass[11pt,a4paper]{article}
\usepackage[utf8]{inputenc}
\usepackage[T1]{fontenc}
\usepackage{mathptmx}
\usepackage{amsfonts}
\usepackage{amsmath}
\usepackage{amsthm}
\usepackage{amssymb}

\usepackage[english]{babel}
\usepackage{amssymb,amsmath}
\usepackage{xcolor}
\usepackage{tcolorbox}
\usepackage{hyperref}
\newtheorem{Theo}{Theorem}[section]
\newtheorem{Prop}[Theo]{Proposition}
\newtheorem{Coro}[Theo]{Corollary}
\newtheorem{Lemm}[Theo]{Lemma}

\theoremstyle{definition}
\newtheorem{Rema}[Theo]{Remark}

\newcommand{\IDT}{\mathbb{T}^{\infty}}

\newcommand{\N}{\mathbb{N}}

\newcommand{\HH}{\mathcal{H}}

\newcommand{\weak}{\operatorname{weak}}

\newcommand{\pr}{\mathfrak{p}}
 
\newcommand{\punkt}{\,\begin{picture}(-1,1)(-1,-3)\circle*{2.5}\end{picture}\;\; }

\newcommand{\unc}{\operatorname{unc}}

\newcommand{\Real}{\operatorname{Re}}

\begin{document}

\title{A note on abscissas of Dirichlet series}

\author{Andreas Defant\thanks{Insitut f\"ur Mathematik. Universit\"at Oldenburg. D-26111 Oldenburg (Germany)  defant@mathematik.uni-oldenburg.de Partially supported by MINECO MTM2017-83262-C2-1-P} 
\and 
Antonio P\'erez 
\thanks{Instituto de Ciencias Matem\'aticas (CSIC-UAM-UC3M-UCM). Campus Cantoblanco UAM, C. Nicol\'as Cabrera 13-15, 28049 Madrid (Spain) antonio.perez@icmat.es Supported by La Caixa Foundation, MINECO MTM2014-57838-C2-1-P and Fundaci\'on S\'eneca - Regi\'on de Murcia (CARM 19368/PI/14)}
\and 
Pablo Sevilla-Peris \thanks{Insituto Universitario de Matem\'atica Pura y Aplicada. Universitat Polit\`ecnica de Val\`encia. Cmno Vera s/n 46022 (Spain) psevilla@mat.upv.es Supported by MINECO MTM2017-83262-C2-1-P}}

\date{\emph{In memory of our friend Bernardo Cascales}}

\maketitle

\begin{abstract}
We present an abstract approach to the abscissas of convergence of vector-valued Dirichlet series. As a consequence we deduce that the abscissas for Hardy spaces of Dirichlet series are all equal. We also introduce and study weak versions of the abscissas for scalar-valued Dirichlet series.
\end{abstract}

\section{Introduction}

Given a complex Banach space $X$, a Dirichlet series with coefficients in $X$ is a formal sum
\[
\sum_{n} a_{n} n^{-s} \,,
\]
where $a_{n} \in X$ for every $n$.  We write $\mathfrak{D}(X)$ for the space of all such Dirichlet series. By a result of Jensen, if $D$ converges at some $s_{0} \in \mathbb{C}$ then it also converges at every point of
\[ \mathbb{C}_{s_{0}}\index{$\mathbb{C}_{0}$} := \{ s \in \mathbb{C} \colon \Real{s} > \Real{s_{0}} \} \,. 
\]
A proof of this fact in the scalar case is given in \cite[p. 97]{QuQu13} or \cite[Theorem~1.1]{DeSe_B}, and it extends in a straightforward way to the vector-valued setting. This means that natural regions of convergence for Dirichlet series are open halfspaces. There are three classical abscissas of convergence for a Dirichlet series $D=\sum a_{n} n^{-s}$
\begin{gather}
\sigma_{a}(D) = \inf{\big\{ \sigma \colon \textstyle\sum a_{n} n^{-s} \text{ converges absolutely on $\mathbb{C}_{\sigma}$ } \big\}} \notag\\
\sigma_{u}(D) = \inf{\big\{ \sigma \colon \textstyle\sum a_{n} n^{-s} \text{ converges uniformly on $\mathbb{C}_{\sigma}$} \big\}} \label{granota}\\
\sigma_{c}(D) = \inf{\big\{ \sigma \colon \textstyle\sum a_{n} n^{-s} \text{ converges pointwise on $\mathbb{C}_{\sigma}$} \big\}} \,, \notag
\end{gather}
H. Bohr in \cite{Bo13_Goett} started to study how far apart can $\sigma_{a}(D) $ and $\sigma_{u}(D)$ be for Dirichlet series with scalar coefficients. This was solved in \cite{BoHi31} (see also \cite[Theorem~4.1]{DeSe_B} or \cite[Theorem~5.4.2]{QuQu13}), where it 
was shown that 
\begin{equation} \label{collserola}
\sup \{ \sigma_{a}(D) - \sigma_{u}(D) \colon D \in \mathfrak{D} (\mathbb{C}) \} = \frac{1}{2} \,. 
\end{equation}
In recent years there has been a renewed interest in this subject. The three classical abscissas have been considered for Dirichlet series with coefficients in some Banach space in \cite{DeGaMaPG08}, and the 
corresponding vector-valued version of Bohr's problem was considered and solved. Also, new abscissas have been considered, as those defined for Hardy spaces of Dirichlet series \cite{Ba02, CaDeSe14}. Our aim in this
note is to present a general approach that allows a unified treatment of all these abscissas, and make a short detour throughout some well-known results on abscissas presenting a different approach with simplified and alternative arguments.

\section{Abscissas of convergence: abstract approach}\label{sec:abscissaUnif}

Some of these classical abscissas have been reformulated in terms of certain Banach spaces of Dirichlet series and, in this way, Functional Analysis has showed to be a useful tool in the study of convergence of scalar and vector valued Dirichlet series. A very good illustration of this interaction is the abscissa of uniform convergence. If we consider the Banach space of Dirichlet series
\begin{multline*}
\mathcal{D}_{\infty} (X) = \big\{ \textstyle\sum a_{n} n^{-s} \, \colon \,  \text{ the limit function }  \text{extends to a holomorphic,} \\  \text{bounded function on } \mathbb{C}_{0} \big\} \,,
\end{multline*}
endowed with the supremum  norm (over $\mathbb{C}_{0}$), then $\sigma_{u}(D)$ can be reformulated as
\begin{equation} \label{vaughan}
\sigma_{u}(D) = \inf \big\{ \sigma \, \colon \, \big( \sum_{n=1}^{N} \frac{a_{n}}{n^{\sigma}} n^{-s} \big)_{N}  \text{ converges in }  \mathcal{D}_{\infty} (X)  \big\} \,.
\end{equation}
A cornerstone of the theory is Bohr's fundamental Theorem (see \cite{QuQu13} for the proof of the scalar case and \cite{DeGaMaPG08} or \cite[Theorem~11.18]{DeSe_B} for the vector-valued version), that can be stated as
\begin{equation} \label{disseratevi}
\sigma_{u}(D) =  \inf \big\{ \sigma \, \colon \,  \sum \frac{a_{n}}{n^{\sigma}} n^{-s} \in   \mathcal{D}_{\infty} (X)  \big\} \,.
\end{equation}

With this in mind, we say that a Banach space $\mathfrak{X}(X)$ of Dirichlet series with coefficients in $X$ is \emph{admissible} if it contains all Dirichlet polynomials (finite sums) in $\mathfrak{D}(X)$ and there are constants $C_{1}, C_{2} > 0$ such that
\begin{equation}\label{equa:admissibleNorm} 
C_{1} \, \sup_{n}{\| a_{n} \|_{X}} \leq \Big\| \sum_{n}{a_{n} n^{-s}} \Big\|_{\mathfrak{X}(X)} \leq C_{2} \, \sum_{n}{\| a_{n} \|_{X}} 
\end{equation}
for every $\sum_{n}{a_{n} n^{-s}}$ in $\mathfrak{X}(X)$. \\
Let us note that the left-hand side inequality in \eqref{equa:admissibleNorm} ensures that if the sequence of partial sums of a Dirichlet series $\sum_{n}{a_{n} n^{-s}}$ converges in $\mathfrak{X}(X)$ then its limit is precisely 
the series itself. In this case we simply say that $\sum_{n}{a_{n} n^{-s}}$ is \textit{convergent} in $\mathfrak{X}(X)$. This can be seen as pointwise convergence under the suitable point of view, namely by considering the canonical inclusion
\[ \mathfrak{X}(X) \longrightarrow \mathfrak{D}(\mathfrak{X}(X))\,, \quad \sum_{n}{a_{n} n^{-s}} \longmapsto \sum_{n}{\widetilde{a_{n}} n^{-z}} \quad \mbox{ where $\widetilde{a_{n}} = a_{n} n^{-s} \in \mathfrak{X}(X)$}. \]
It is then clear that $\sum_{n}{a_{n} n^{-s}}$ is convergent in $\mathfrak{X}(X)$ if and only if $\sum_{n}{\widetilde{a_{n}} n^{-z}}$ converges pointwise at $z=0$. This simple observation yields, as a consequence of the vector-valued  version of Jensen's result mentioned in the introduction, that if $\sum_{n}{a_{n}n^{-s}}$ is convergent in $\mathfrak{X}(X)$, then $\sum_{n}{\frac{a_{n}}{n^{z}}n^{-s}}$ is also convergent in $\mathfrak{X}(X)$ whenever $\Real{(z)} > 0$. It is thus reasonable to define the abscissa of $\mathfrak{X}(X)$-convergence of $D=\sum_{n}{a_{n}n^{-s}}  \in \mathfrak{D}(X)$ as
\begin{equation} \label{mahogany}
\sigma_{\mathfrak{X}(X)}(D)  = \inf  \big\{ \sigma \in \mathbb{R} \colon \sum_{n} \frac{a_{n}}{n^{\sigma}} n^{-s} \text{ converges in } \mathfrak{X}(X)   \big\} \,
\end{equation}
where we agree to write $\inf{(\varnothing)}:=+\infty$ and $\inf{(\mathbb{R})} = - \infty$. 

\begin{Rema} \label{stlouis}
The classical abscissas from \eqref{granota} can be reformulated in this new language. To begin with, as we already noted in \eqref{vaughan} $\sigma_{u}$ is $\sigma_{\mathcal{D}_{\infty}(X)}$. Also, if we define $c(X)$ as the 
space of Dirichlet series that converge at $0$ with the norm $\big\Vert \sum a_{n} n^{-s} \big\Vert_{c(X)} = \sup_{N}{\big\Vert \sum_{n=1}^{N} a_{n} \big\Vert_{X}}$, then $\sigma_{c}$ is $\sigma_{c(X)}$. Finally, the space of all Dirichlet series with absolute summable coefficients (that we denote $\ell_{1}(X)$) endowed with the norm $\big\| \sum a_{n} n^{-s} \big\|_{\ell_{1}(X)} = \sum_{n=1}^{\infty} \| a_{n}\|_{X}$  gives the abscissa of absolute convergence.
\end{Rema}

A useful tool for scalar valued Dirichlet series are the Bohr-Cahen formulas (see \cite[Proposition~1.6]{DeSe_B} or \cite[Theorem~4.2.1]{QuQu13}). A careful analysis of the proof shows how to extend these formulas to our new abscissa.

\begin{Prop}\label{Prop:BasicEstimationsAbscissae}
Let $\mathfrak{X}(X)$ be an admissible Banach space of Dirichlet series with coefficients in a Banach space $X$. For every $D = \sum_{n}{a_{n}n^{-s}}$ in $\mathfrak{D}(X)$ we have then that:\\[-3mm]
\begin{gather}
\label{equa:abscissaPartialBounds}  \sigma_{\mathfrak{X}(X)}(D) = \inf{\{ \sigma \in \mathbb{R} \colon \mbox{$\sup_{N} \big\| \sum_{n=1}^{N}{\frac{a_{n}}{n^{\sigma}} n^{-s}}\big\|_{\mathfrak{X}(X)} < \infty$} \}}\\[2mm]
\label{BohrCahenGeneral} \sigma_{\mathfrak{X}(X)}(D) \leq \limsup_{N}{\frac{\log{\left\| \sum_{n=1}^{N}{a_{n} n^{-s}} \right\|_{\mathfrak{X}(X)} }}{\log{N}}  } 
\end{gather}
being the last relation an equality whenever the abscissa is non-negative. 
\end{Prop}

\begin{proof}
Both statements follow from Abel's summation formula, used in each case in a slightly different way. For the first one, let us choose $\sigma \in \mathbb{R}$ satisfying
\[ 
C:=\sup_{N}{\Big\| \sum_{n=1}^{N}{\frac{a_{n}}{n^{\sigma}} n^{-s}} \Big\|_{\mathfrak{X}(X)}} < \infty \,. 
\]
Now Abel's summation formula yields, for every $1 \leq M < N-1$ and every $\varepsilon > 0$,
\[
\sum_{n=M+1}^{N}{\frac{a_{n}}{n^{\sigma + \varepsilon+s}}} = \Big( \sum_{k=1}^{N}{\frac{a_{k}}{k^{\sigma+s}}} \Big) \frac{1}{N^{\varepsilon}} -  \Big( \sum_{k=1}^{M}{\frac{a_{k}}{k^{\sigma+s}}} \Big) \frac{1}{M^{\varepsilon}}+ \sum_{n=M}^{N-1}\Big( \sum_{k=1}^{n}{\frac{a_{k}}{k^{\sigma+s}}} \Big) \Big[\frac{1}{n^{\varepsilon}} - \frac{1}{(n+1)^{\varepsilon}}\Big] \,. 
\]
Taking norms in the previous expression we get
\[
\Big\| \sum_{n=M+1}^{N}{\frac{a_{n}}{n^{\sigma + \varepsilon}} n^{-s}} \Big\|_{\mathfrak{X}(X)} \leq \frac{C}{N^{\varepsilon}} + \frac{C}{M^{\varepsilon}} + \sum_{n=M}^{N-1}{C \Big[ \frac{1}{n^{\varepsilon}} - \frac{1}{(n+1)^{\varepsilon}}\Big]} = \frac{2 C}{M^{\varepsilon}}
\]
and thus $\sigma_{\mathfrak{X}(X)}(D) \leq \sigma + \varepsilon$. This proves 
\[
\sigma_{\mathfrak{X}(X)}(D) \leq \inf{\{ \sigma \in \mathbb{R} \colon \mbox{$\sup_{N} \big\| \sum_{n=1}^{N}{\frac{a_{n}}{n^{\sigma}} n^{-s}}\big\|_{\mathfrak{X}(X)} < \infty$} \}} \,,
\]
The reverse inequality is clear, and this gives \eqref{equa:abscissaPartialBounds}.\\
For the second statement, let $L$ be the superior limit in \eqref{BohrCahenGeneral}, which we can assume to be finite. Using Abel's summation formula in a slightly different way, 
\[
\sum_{n=M+1}^{N}{\frac{a_{n}}{n^{L + \varepsilon+s}}} = \Big( \sum_{k=1}^{N}{\frac{a_{k}}{k^{s}}} \Big) \frac{1}{N^{L+\varepsilon}} -  \Big( \sum_{k=1}^{M}{\frac{a_{k}}{k^{s}}} \Big) \frac{1}{M^{L+\varepsilon}}+ \sum_{n=M}^{N-1}\Big( \sum_{k=1}^{n}{\frac{a_{k}}{k^{s}}} \Big) \Big[\frac{1}{n^{L+\varepsilon}} - \frac{1}{(n+1)^{L+\varepsilon}}\Big] \,
\]
for every $1 \leq M < N-1$ and $\varepsilon >0$. 
By the very definition of $L$, we can find $n_{0}$ so that, for all $n \geq n_{0}$, we have
$\big\Vert \sum_{k=1}^{n} a_{n} k^{-s} \big\Vert_{\mathfrak{X}(X)} \leq n^{L+ \varepsilon/2}$. Therefore
\[
\Big\Vert \sum_{n=M+1}^{N} \frac{a_{n}}{n^{L+\varepsilon}} n^{-s} \Big\Vert_{\mathfrak{X}(X)}
\leq \frac{1}{N^{\frac{\varepsilon}{2}}} +  \frac{1}{M^{\frac{\varepsilon}{2}}}
+ \vert L + \varepsilon \vert \sum_{n=M}^{N-1} \frac{1}{n^{1+\frac{\varepsilon}{2}}} \,
\]
for all $N>M +1 >n_{0}$.
Obviously $n_{0}$ can be chosen to make the last term as small as we want. This shows that the series converges in $\mathfrak{X}(X)$ and completes the proof of \eqref{BohrCahenGeneral}.\\
Finally, if $\sigma_{\mathfrak{X}(X)}(D) \geq 0$, we take $\varepsilon >0$ and set $\sigma_{0} = \sigma_{\mathfrak{X}(X)}(D) + \varepsilon$. Then $\big( \sum_{n=1}^{N}  \frac{a_{n}}{n^{\sigma_{0}}} n^{-s} \big)_{N}$ is convergent in $\mathfrak{X}(X)$ and hence norm bounded by, say, $B>0$. Abel's summation (once again) gives
\[
\sum_{n=1}^{N} a_{n} n^{-s} = \Big(\sum_{n=1}^{N} \frac{a_{n}}{n^{\sigma_{0}+s}} \Big) N^{\sigma_{0}}
+ \sum_{n=1}^{N-1} \Big( \sum_{k=1}^{n} \frac{a_{k}}{k^{\sigma_{0}+s}} \Big) \big( n^{\sigma_{0}} - (n+1)^{\sigma_{0}} \big)
\]
and
\[
\Big\Vert \sum_{n=1}^{N} a_{n} n^{-s} \Big\Vert_{\mathfrak{X}(X)} \leq 2 B N^{\sigma_{0}} \,.
\]
This implies $L\leq \sigma_{0}=\sigma_{\mathfrak{X}(X)}(D) + \varepsilon$ and completes the proof.
\end{proof}

\subsection{Strips between abscissas}

Let $\mathfrak{X}_{1}(X)$, $\mathfrak{X}_{2}(X)$ be admissible Banach spaces of Dirichlet series. If there is a constant $C>0$ such that $\| P\|_{\mathfrak{X}_{2}(X)} \leq C \| P\|_{\mathfrak{X}_{1}(X)}$ for every Dirichlet polynomial $P$ of $\mathfrak{D}(X)$, then by Proposition \ref{Prop:BasicEstimationsAbscissae} we easily deduce
\[ 
\sigma_{\mathfrak{X}_{2}(X)}(D) \leq \sigma_{\mathfrak{X}_{1}(X)}(D) \quad \mbox{ for every } D \in \mathfrak{X}_{2}(X). \, 
\]
Then it is reasonable to ask for the exact value of the maximum possible width of the strip determined by both abscissas, namely 
\[ 
S(\mathfrak{X}_{1}(X), \mathfrak{X}_{2}(X)) :=
\sup{\{ \sigma_{\mathfrak{X}_{1}(X)}(D) - \sigma_{\mathfrak{X}_{2}(X)}(D)\colon D \in \mathfrak{D}(X) \}} \,. 
\]
Notice that, from the definition of admissible space of Dirichlet series, we have (here $\ell_{\infty}(X)$ denotes the space of bounded sequences in $X$)
\[ \
\sigma_{\ell_{\infty}(X)}(D) \leq \sigma_{\mathfrak{X}(X)}(D) \leq \sigma_{\ell_{1}(X)}(D)  \quad \mbox{ for every } D \in \mathfrak{X}(X). 
\]
As a straightforward consequence, we have that the abscissas of convergence are all simultaneously infinite or finite, being in the last case $\sigma_{\ell_{1}(X)} (D) \leq \sigma_{\ell_{\infty}}(D) + 1$.\\
As we already noted in \eqref{collserola}, Bohr's original problem (finally solved by Bohnenblust and Hille) was to determine the maximal distance between the abscissas of absolute and uniform convergence of Dirichlet series with scalar coefficients:
\[
S(\ell_{1}(\mathbb{C}), \mathcal{D}_{\infty}(\mathbb{C})) = \sup{\{ \sigma_{a}(D) - \sigma_{u}(D) \colon D \in \mathfrak{D}(\mathbb{C})\}} = \frac{1}{2} \,. 
\]
This was extended in \cite{DeGaMaPG08}, where the corresponding problem was considered for Dirichlet series with coefficients in some Banach space $X$, showing that
\begin{equation} \label{tempestade}
S(\ell_{1}(X), \mathcal{D}_{\infty}(X)) = 1 - \frac{1}{\cot(X)}  \,
\end{equation} 
where $\cot{(X)}$ is the infimum of all  $q \geq 2$ such that $X$ has cotype $q$ if $X$ has finite cotype, and $\cot{(X)}=\infty$ if $X$ has no finite cotype (see e.g. \cite{DiJaTo95} for the definition of cotype).

We aim now a useful result for the comparison of abscissas which extends a previous result of Maurizi and Queff\'{e}lec \cite[Theorem~2.4]{MaQu10}. First we need to introduce a new condition inspired by \eqref{disseratevi}: we say that an admissible Banach space of Dirichlet series $\mathfrak{X}(X)$ satisfies \emph{Bohr's fundamental Theorem} if for every $D \in \mathfrak{X}(X)$ it holds that $\sigma_{\mathfrak{X}(X)}(D) \leq 0$, or equivalently, if
\[ 
\sigma_{\mathfrak{X} (X)}(D) = \inf \big\{ \sigma \in \mathbb{R} \colon \sum \frac{a_{n}}{n^{\sigma}} n^{-s} \in \mathfrak{X}(X) \big\} \,
\]
for every $D=\sum a_{n}n^{-s} \in\mathfrak{X} (X)$.

\begin{Prop}\label{Prop:stripMauQueff}
Let $\mathfrak{X}_{1}(X), \mathfrak{X}_{2}(X)$ be two admissible Banach spaces of Dirichlet series with coefficients in $X$, where $ \mathfrak{X}_{2}(X)$ satisfies Bohr's fundamental Theorem. Then, for every $\delta > 0$ the following assertions are equivalent:
\begin{enumerate}
\item[(i)] \label{MaQu1} $\sigma_{\mathfrak{X}_{1}(X)}(D) \leq \sigma_{\mathfrak{X}_{2}(X)}(D) + \delta$ for every Dirichlet series $D$ in $\mathfrak{D}(X)$.\\[-2mm]
\item[(ii)] \label{MaQu2} For each $\sigma > \delta$ there is $C_{\sigma} > 0$ such that, for every $a_{1}, \ldots, a_{N} \in X$ and $N \in \N$
\begin{equation}\label{equa:polynomialStrip}
\Big\| \sum_{n=1}^{N} a_{n} n^{-s} \Big\|_{\mathfrak{X}_{1}(X)} \leq C_{\sigma}\, N^{\sigma} \, \Big\| \sum_{n=1}^{N} a_{n}n^{-s} \Big\|_{\mathfrak{X}_{2}(X)}\,.
\end{equation}
\end{enumerate}
\end{Prop}

\begin{proof}
(i) $\Rightarrow$ (ii): Fixed $\sigma > \delta$, consider the sequence of bounded operators
\[ 
P_{N}: \mathfrak{X}_{2}(X) \longrightarrow \mathfrak{X}_{1}(X)\,, \quad \sum_{n}{a_{n} n^{-s}} \longmapsto \sum_{n \leq N}{\frac{a_{n}}{n^{\sigma}} n^{-s}} 
\]
For each $D \in \mathfrak{X}_{2}(X)$ we have that $\sigma_{\mathfrak{X}_{2}(X)}(D) \leq 0$ as  $\mathfrak{X}_{2}(X)$ satisfies Bohr's fundamental Theorem. Using (i) this yields that $\sigma_{\mathfrak{X}_{1}(X)}(D) < \sigma$, so $P_{N}(D)$ is convergent and hence bounded. Applying the uniform boundedness principle we deduce that $C_{\sigma}:= \sup_{N}{\| P_{N}\|} < \infty$, so that for every $\sum_{n}{a_{n} n^{-s}}$ in $\mathfrak{X}_{2}(X)$ it holds that
\[ 
\Big\| \sum_{n \leq N}{\frac{a_{n}}{n^{\sigma}} n^{-s}} \Big\|_{\mathfrak{X}_{1}(X)} \leq C_{\sigma} \Big\| \sum_{n}{a_{n} n^{-s}} \Big\|_{\mathfrak{X}_{2}(X)} \,. 
\]
Finally, using Abel's summation formula we get
\begin{multline*} 
\Big\| \sum_{n = 1}^{N}{\frac{a_{n}}{n^{s}}} \Big\|_{\mathfrak{X}_{1}(X)}  \leq \sum_{n=1}^{N-1}{\Big\| \sum_{k=1}^{n}{\frac{a_{k}}{k^{\sigma + s}}} \Big\|_{\mathfrak{X}_{1}(X)} [(n+1)^{\sigma} - n^{\sigma}]} + \Big\| \sum_{k=1}^{N}{\frac{a_{k}}{k^{\sigma + s}}}\Big\|_{\mathfrak{X}_{2}(X)} \, N^{\sigma}\\
 \leq C_{\sigma} \, \Big\|  \sum_{n=1}^{N}{\frac{a_{n}}{n^{s}}} \Big\|_{\mathfrak{X}_{2}(X)}  \sum_{n=1}^{N-1} \big[(n+1)^{\sigma} - n^{\sigma} \big] + N^{\sigma} \leq 2 \, C_{\sigma}  \, N^{\sigma} \, \Big\|  \sum_{n=1}^{N}{\frac{a_{n}}{n^{s}}} \Big\|_{\mathfrak{X}_{2}(X)} \,.
\end{multline*}
Assume now that \eqref{equa:polynomialStrip} holds and fix $D=\sum_{n}{a_{n}n^{-s}}$ in $\mathfrak{D}(X)$ with finite abscissas. Suppose first that $\sigma_{\mathfrak{X}_{2}(X)}(D) \geq 0$. Taking logarithm on both sides of \eqref{equa:polynomialStrip} it immediately follows that
\[ 
\frac{1}{\log{N}} \log \Big\| \sum_{n=1}^{N}{\frac{a_{n}}{n^{s}}} \Big\|_{\mathfrak{X}_{1}(X)} \leq \frac{\log{C_{\sigma}}}{\log{N}} + \sigma + \frac{1}{\log{N}} \log \Big\| \sum_{n=1}^{N}{\frac{a_{n}}{n^{s}}} \Big\|_{\mathfrak{X}_{2}(X)} 
\]
and Proposition~\ref{Prop:BasicEstimationsAbscissae} immediately gives 
\[
\sigma_{\mathfrak{X}_{1}(X)}(D) \leq \sigma_{\mathfrak{X}_{2}(X)}(D) + \sigma \,.
\] 
Finally, if $\sigma_{\mathfrak{X}_{2}(X)}(D) < 0$ we consider $\alpha \in \mathbb{R}$ with $\alpha < \sigma_{\mathfrak{X}_{2}(X)}(D)$ so that $\sigma_{\mathfrak{X}_{2}(X)}\big( \sum \frac{a_{n}}{n^{\alpha}} n^{-s} \big) = \sigma_{\mathfrak{X}_{2}(X)}(D) - \alpha > 0$. Applying the previous argument to the series  $\sum \frac{a_{n}}{n^{\alpha}} n^{-s} $ we immediately get the result.
\end{proof}

As a consequence of the previous proposition, if $\mathfrak{X}_{1}(X)$ and $\mathfrak{X}_{2}(X)$ are admissible Banach spaces such that  $\sigma_{\mathfrak{X}_{2}(X)} (D) \leq \sigma_{\mathfrak{X}_{1}(X)}(D)$ for every $D \in \mathfrak{D}(X)$, then we can write
\begin{multline*}
S(\mathfrak{X}_{1}(X), \mathfrak{X}_{2}(X))\\ =
\inf \big\{ \sigma \in \mathbb{R} \colon  \exists C_{\sigma} \text{ s.t. } \Big\| \sum_{n=1}^{N} a_{n} n^{-s} \Big\|_{\mathfrak{X}_{2}(X)} \leq C_{\sigma}\, N^{\sigma} \, \Big\| \sum_{n=1}^{N} a_{n}n^{-s} \Big\|_{\mathfrak{X}_{1}(X)} \\
\text{ for all } a_{1}, \ldots, a_{N} \in X \text{ and } N \in \mathbb{N} \big\}
\end{multline*}

\subsection{Abscissas of Hardy spaces}

Hardy spaces of Dirichlet series were first considered in \cite{Ba02} for series with scalar coefficients and in \cite{CaDeSe14} for series with coefficients in some Banach space. A way to define these spaces is through Hardy spaces of vector-valued functions. Let us roughly sketch how this is done. We denote by $\mathbb{N}_{0}^{(\mathbb{N})}$ the set of all finite sequences of elements in $\N_{0}:=\N \cup \{ 0\}$. Then, to each $\alpha = (\alpha_{1}, \ldots, \alpha_{N}) \in \mathbb{N}_{0}^{(\mathbb{N})}$ corresponds a (unique) natural number $n = \pr_{1}^{\alpha_{1}} \cdots \pr_{N}^{\alpha_{N}} =: \pr^{\alpha}$ (here $\pr=(\pr_{k})$ denotes the sequence of prime numbers). With this idea we define a bijection between $\mathfrak{P}(X)$ (the space of all formal power series in infinitely many variables and with coefficients in $X$) and $\mathfrak{D}(X)$ 
\begin{equation} \label{araceli}
\begin{array}{cccc}
\mathfrak{B}_{X} : \quad & \,\,\mathfrak{P}(X) & \xrightarrow{\text{\hspace{2cm}}} & \mathfrak{D} (X)\, \\
& \,\, \textstyle\sum c_{\alpha} z^{\alpha} \,\, &
\xrightarrow{\;a_{\pr^{\alpha}} = c_{\alpha}\;}\,\, &
\textstyle\sum a_{n} n^{-s}
\end{array}
\end{equation}
that we call \textit{Bohr transform}. By $\mathbb{T}^{\infty}$ we denote the (infinite) countable product of the torus $\mathbb{T}$ with itself, which is again a compact group. Now, given a Banach space $X$ and $1 \leq p \leq \infty$ we consider $H_{p}(\mathbb{T}^{\infty},X)$, the (closed) subspace of $L_{p}(\mathbb{T}^{\infty},X)$ consisting of those 
functions whose Fourier coefficients $\widehat{f}(\alpha)$ are different from zero only if $\alpha \in \mathbb{N}_{0}^{(\mathbb{N})}$. Associating to each $f \in H_{p}(\mathbb{T}^{\infty},X)$ the formal power series $\sum_{\alpha \in \mathbb{N}_{0}^{(\mathbb{N})}} \widehat{f}(\alpha) z^{\alpha}$, this can be seen as a subspace of $\mathfrak{P}(X)$ and we define the Hardy space of Dirichlet series as (see \cite{CaDeSe14} or \cite{DePe16} for more details)
\[
\mathcal{H}_{p} (X) = \mathfrak{B}_{X} ( H_{p} (\mathbb{T}^{\infty} , X)  ) 
\]
with the norm induced by the identification. This is an admissible Banach space of Dirichlet series and the corresponding abscissa \eqref{mahogany} was already defined in \cite{CaDeSe14}. Our aim now is to show that, in 
fact, all these abscissas are equal for all $1 \leq p < \infty$ (Theorem~\ref{hypercontractivityabs}). This is going to follow from an independently interesting result (Theorem~\ref{geyrenhoff}): if we translate a Dirichlet series in some $\mathcal{H}_{p}(X)$ a little bit to the right, then it
belongs to every $\mathcal{H}_{q}(X)$ for every $p<q$. The arguments are basically the same as the ones for the scalar case that can be found in \cite{Ba02}. \\
We reformulate our problem in terms of functions on $\mathbb{T}^{\infty}$. We begin with a simple observation in the one dimensional torus. Given $0 \leq r <1$ and $1 \leq p, q \leq \infty$, we have a bounded linear operator
\begin{equation} \label{kakabsa}
T_{r}: H_{p}(\mathbb{T}, X) \longrightarrow H_{q}(\mathbb{T}, X) \,, \quad \sum_{n \geq 0}{\widehat{f}(n) z^{n}} \longmapsto \sum_{n \geq 0}{\widehat{f}(n) r^{n}z^{n}} 
\end{equation}
with norm $\| T_{r}\| \leq \frac{1}{1-r}$. The following proposition extends \eqref{kakabsa}  to the infinite dimensional torus for finite values of $p$.

\begin{Prop} \label{ruthmayer}
Let $r = (r_{k}) \in [0,1)^{\N}$  and $1 \leq p \leq q < \infty$ be such that the operators $T_{r_{k}}$ defined in \eqref{kakabsa} satisfy $C_{r}:=\sup_{n} \prod_{k=1}^{n} \Vert T_{r_{k}} \Vert < \infty$. Then, there exists a unique operator 
\[ 
T_{r} : H_{p} (\mathbb{T}^{\infty}, X) \longrightarrow H_{q} (\mathbb{T}^{\infty}, X)
\]
such that $T_{r} \big(\sum_{\alpha} c_{\alpha} w^{\alpha} \big) = \sum_{\alpha} c_{\alpha} (rw)^{\alpha}$ for every polynomial, which moreover satisfies that $\Vert T_{r} \Vert \leq C_{r}$.
\end{Prop}
\begin{proof}
Let $P = \sum_{\alpha \in \N_{0}^{n}}{c_{\alpha} w^{\alpha}}$ be an analytic polynomial in $n$ variables. Using Fubini's theorem and the integral Minkowski inequality it is easy to see that
\[ \left( \int_{\mathbb{T}^{N}}{\| P(z_{1}, \ldots, r_{k} z_{k}, \ldots z_{n})\|_{X}^{q}} d z \right)^{\frac{1}{q}} \leq \| T_{r_{k}}\|  \left( \int_{\mathbb{T}^{N}}{\| P(z_{1}, \ldots, z_{k}, \ldots z_{n})\|_{X}^{p}} d z \right)^{\frac{1}{p}}. \]
Iterating this inequality in each variable, we deduce that
\[ 
\|T_{r} P\|_{H_{q}(\mathbb{T}^{\infty},X)} \leq \Big( \prod_{k=1}^{n}{\| T_{r_{k}}\|} \Big) \, \| P\|_{H_{p}(\mathbb{T}^{\infty}, X)} \leq C_{r} \, \| P\|_{H_{p}(\mathbb{T}^{\infty}, X)}.
\]
Since analytic polynomials are dense in $H_{q}(\mathbb{T}^{\infty}, X)$ ($q$ is finite), the map $T_{r}$ can be extended to a bounded linear operator from $H_{q}(\mathbb{T}^{N}, X)$ into $H_{p}(\mathbb{T}^{N}, X)$ with $\| T_{r}\| \leq C_{r}$.
\end{proof}

A deep result of Weissler, whose vector-valued version was established in \cite[Lemma~1.3]{CaDeSe16}, provides for $1 \leq p \leq q < \infty$ with the bound 
\begin{equation}\label{equa:WeisslerResult}
\|T_{r} : H_{p} (\mathbb{T},X) \longrightarrow H_{q} (\mathbb{T},X)\|\leq 1 \quad \mbox{ whenever } \quad 0 \leq r \leq \sqrt{p/q}\,.
\end{equation}
We use this fact to prove the result we are aiming at.

\begin{Theo}\label{geyrenhoff}
Let $1 \leq p \leq q < \infty$. For every $\sum a_{n} n^{-s}$ in $\mathcal{H}_{q}(X)$ and every $\varepsilon > 0$, the series $\sum \frac{a_{n}}{n^{\varepsilon}} n^{-s}$ belongs to $\mathcal{H}_{p}(X)$.
\end{Theo}
\begin{proof}
Let us in a first stage take some that $0 < r_{k} < 1$ as in the statement of Proposition~\ref{ruthmayer}, and define for each $n=\pr^{\alpha} \in \mathbb{N}$ the element $\lambda_{n} = r^{\alpha}$. Our aim now is to show that the
map $M_{\lambda} : \mathcal{H}_{p}(X) \longrightarrow \mathcal{H}_{q}(X)$ which associates $\sum a_{n} n^{-s} \longmapsto \sum \lambda_{n} a_{n} n^{-s}$ is a well defined continuous operator. A simple computation shows that
$M_{\lambda}$ and $\mathfrak{B}_{X} T_{r} \mathfrak{B}^{-1}_{X}$ (where $\mathfrak{B}_{X}$ is the Bohr transform from \eqref{araceli}) coincide on the Dirichlet polynomials. Then, a density argument shows that $M_{\lambda}$ is well defined, bounded and satisfies  the norm estimate
\[
\Vert M_{\lambda} \Vert = \Vert \mathfrak{B} T_{r} \mathfrak{B}_{X}^{-1}  \Vert = \Vert T_{r} \Vert \leq \sup_{n} \prod_{k=1}^{n} \Vert T_{r_{k}} \Vert \,.
\]
With this observation we can finish the proof of the theorem. Fixed $\varepsilon >0$, define $r_{k} = \pr_{k}^{-\varepsilon}$ for $k \in \mathbb{N}$. Then, there is some $k_{0}$ such that for all $k \ge k_0$ it holds $r_{k} \leq \sqrt{p/q}$, so that $\Vert T_{r_{k}} \Vert \leq 1$ by \eqref{equa:WeisslerResult}. In particular, $\sup_{n} \prod_{k=1}^{n} \Vert T_{r_{k}} \Vert  =  \prod_{k=1}^{k_{0}} \Vert T_{r_{k}} \Vert $ is obviously finite and we can apply the previous observation to $\lambda_{n}= \frac{1}{n^{\varepsilon}}$ to conclude that the operator $\mathcal{H}_{p}(X) \longrightarrow \mathcal{H}_{q}(X)$ given by $\sum a_{n} n^{-s} \longmapsto \sum \frac{a_{n}}{n^{\varepsilon}} n^{-s}$ is well defined and continuous.
\end{proof}

As the Banach space $\mathcal{H}_{p}(X)$ satisfies Bohr's fundamental theorem (see \cite[Corollary~3.3]{DePe16}), we immediately deduce from Theorem~\ref{geyrenhoff} the following result.

\begin{Theo} \label{hypercontractivityabs}
Given $1 \leq  p \leq q < \infty$, every Dirichlet series $D= \sum a_n n^{-s}$ in $\mathfrak{D}(X)$ satisfies that
$\sigma_{\mathcal{H}_p (X)}(D) = \sigma_{\mathcal{H}_q(X)}(D)$\,.
\end{Theo}

An alternative proof of Theorem \ref{hypercontractivityabs} can be carried out by means of the following result stated in \cite[Remark I]{DePe17}: the best constant $\mho(q,p,N)$ satisfying
\[ \Big\| \sum_{n= 1}^{N}{\frac{a_{n}}{n^{s}}} \Big\|_{\mathcal{H}_{q}(X)} \leq \mho(q,p, N) \Big\| \sum_{n = 1}^{N}{\frac{a_{n}}{n^{s}}} \cdot \Big\|_{\mathcal{H}_{p}(X)} \]
for every $a_{1}, \ldots, a_{N} \in X$ can be estimated asymptotically on $N$ as
\begin{equation}\label{equa:MoinConstant} 
\mho(q,p, N)= \exp{\left( \frac{\log{N}}{\log{\log{N}}} \left( \log{\sqrt{\frac{q}{p}}} + o(1) \right) \right)}.
\end{equation}
This yields in particular that $\mho(q,p, N) = O(N^{\varepsilon})$ for every $\varepsilon > 0$, so it follows from Proposition \ref{Prop:stripMauQueff} that $\sigma_{\HH_{q}(X)}(D) \leq \sigma_{\HH_{p}(X)}(D)$ for every Dirichlet series $D$ in $\mathfrak{D}(X)$. The reverse inequality is obvious. The proof of \eqref{equa:MoinConstant} relies again on the aforementioned (vector-valued) result of Weissler \eqref{equa:WeisslerResult} combined with tools from number theory.\\

Bohr's absolute problem for $\mathcal{H}_{p}(X)$-spaces has been studied, firstly by Bayart in \cite{Ba02} for series with coefficients in $\mathbb{C}$, and later in \cite{CaDeSe14} with coefficients in some arbitrary Banach space (see also \cite{DeSe_B, QuQu13}), showing that
\begin{equation} \label{novos}
S(\ell_{1}(X), \mathcal{H}_{p}(X)) = \sup \{  \sigma_{a}(D) - \sigma_{\mathcal{H}_{p}(X)}(D)  \,\colon \, D \in \mathfrak{D} (X) \} = 1 - \frac{1}{\cot(X)}  \,.
\end{equation}
The lower bound of this estimation is consequence of \eqref{tempestade} and  the fact that uniform convergence is stronger that $\mathcal{H}_{p}(X)$-convergence, which yield that
\[ 1- \frac{1}{\cot{(X)}} = S(\ell_{1}(X), \mathfrak{D}_{\infty}(X)) \leq S(\ell_{1}(X), \mathcal{H}_{p}(X)). \]
We give now an alternative new proof of the upper estimate in \eqref{novos} (shorter than the original) which is also based on \eqref{tempestade}. We just have to prove it for the case $p=1$. It was established in \cite[Lemma~3.3]{DePe16} that the map
\[ \mathcal{H}_{1}(X) \longrightarrow \mathfrak{D}_{\infty}(\mathcal{H}_{1}(X)) \,, \quad D=\sum_{n}{a_{n}n^{-s}} \longmapsto \widetilde{D}=\sum_{n}{(a_{n} n^{-s}) n^{-z}} \]
is well defined and establishes a canonical isometric isomorphism from $\mathcal{H}_{1}(X)$ into a subspace of $\mathfrak{D}_{\infty}(\mathcal{H}_{1}(X))$, so that
\[ \sigma_{\mathfrak{D}_{\infty}(\mathcal{H}_{1}(X))}(\widetilde{D}) = \sigma_{\mathcal{H}_{1}(X)}(D). \]
Therefore, using \eqref{tempestade}
\[ S(\ell_{1}(X), \mathcal{H}_{1}(X)) \leq S(\ell_{1}(X), \mathfrak{D}_{\infty}(\mathcal{H}_{1}(X))) = 1 - \frac{1}{\cot{(\mathcal{H}_{1}(X))}}, \]
and finally, since $\cot{(L_{1}(\mu, X))} = \cot{(X)}$ (see \cite[Theorem 11.12]{DiJaTo95})
\[\cot{(\mathcal{H}_{1}(X))} = \cot{(H_{1}(\IDT,X))} \leq \cot{(L_{1}(\mathbb{T}^{\infty}, X))} = \cot{(X)}.\]
\section{Weak abscissas}

Given a Dirichlet series $ D= \sum a_{n} n^{-s}$ taking values in a Banach space $X$ and a functional $x^{*} \in X^{*}$ we have a scalar valued Dirichlet series
\[
D_{x^{\ast}} = \sum x^{*}( a_{n} ) n^{-s} \,.
\]
Based on this, if we have a well defined abscissa of convergence $\sigma_{\punkt}(\cdot)$ for scalar Dirichlet series, then one can introduce the following \emph{weak-abscissas of convergence} for a Dirichlet series $D$ with coefficients in $X$
\begin{equation}\label{equa:weakAbscissaOriginal}
\sigma_{\punkt}^{\weak}(D) = \sup_{\Vert x^\ast \Vert_{X^{*}} < 1} \sigma_{\punkt}(D_{x^{\ast}})  \,.
\end{equation}
These abscissas were introduced by J. Bonet \cite{Bo15} in the more general setting of locally convex spaces for the cases of pointwise, uniform and absolute convergence, namely $\sigma_{c}^{\weak}$, $\sigma_{u}^{\weak}$ and $\sigma_{a}^{\weak}$, and studied its relation with the vector-valued versions $\sigma_{c}$, $\sigma_{u}$ and $\sigma_{a}$ respectively. \\
We are going to show that in general every weak-abscissa of convergence can be actually been as an abscissa of convergence for a suitable admissible space, which will let us deduce some immediate consequences. To do that, given $\mathfrak{X}$ an admissible Banach space of scalar Dirichlet series and a Banach space $X$ we define the following space of
Dirichlet series with coefficients in $X$:
\begin{equation} \label{rascal}
\mathfrak{X}^{\weak}(X):= \Big\{ \sum a_{n}n^{-s} \in \mathfrak{D}(X) \colon \sup_{x^{\ast} \in X^{\ast}} \Big\| \sum x^{*} (a_{n})  n^{-s} \Big\|_{\mathfrak{X}} < \infty \Big\} 
\end{equation}
endowed with the norm $\big\| \sum a_{n}n^{-s} \big\|_{\mathfrak{X}^{\weak}(X)} := \sup_{x^{\ast} \in X^{\ast}} \big\| \sum x^{*} (a_{n})  n^{-s} \big\|_{\mathfrak{X}}$

\begin{Prop}
Let $\mathfrak{X}$ be an admissible Banach space of scalar Dirichlet series, then the space $\mathfrak{X}^{\weak}(X)$ defined in \eqref{rascal} is an admissible Banach space and
\begin{equation} \label{basinstreet}
\sigma^{\weak}_{\mathfrak{X}}(D) = \sigma_{\mathfrak{X}^{\weak}(X)}(D) 
\end{equation}
for every Dirichlet series with coefficients in $X$.
\end{Prop} 

\begin{proof}
It is clear that $\mathfrak{X}^{\weak}(X)$ is a well defined normed space containing all Dirichlet polynomials on $X$ and whose norm satisfies \eqref{equa:admissibleNorm}. The argument for completeness is standard.  
If $D_{k} = \sum_{n}{a_{n}^{k} n^{-s}}$ is a Cauchy sequence in $\mathfrak{X}^{\weak}(X)$, then by \eqref{equa:admissibleNorm} we have that for each $n$ the sequence $(a_{n}^{k})_{k}$ is Cauchy in $X$ and 
hence convergent to some $a_{n}$. Consider now the Dirichlet series $D=\sum a_{n} n^{-s}$. For each $x^{*} \in X^{*}$ the sequence  $\big(  \sum_{n} x^{*}(a_{n}^{k}) n^{-s} \big)_{k}$ is Cauchy in  $\mathfrak{X}$.
It is thus convergent and, since $\lim_{k}{x^{\ast}(a_{n}^{k})} = x^{\ast}(a_{n})$ for each $n$, its limit must be $\sum x^{*}(a_{n}) n^{-s}$. In other words, we have that the Cauchy sequence of operators 
\[ X^{\ast} \longrightarrow \mathfrak{X}\,, \quad x^{\ast} \longmapsto \sum_{n}{x^{\ast}(a_{n}^{k})n^{-s}}\] 
converges pointwise to the map 
\[ X^{\ast} \longrightarrow \mathfrak{X}\,, \quad x^{\ast} \longmapsto \sum_{n}{x^{\ast}(a_{n})n^{-s}}\]
Hence, this last map is a bounded linear operator (i.e. $D \in \mathfrak{X}^{\weak}(X)$) and the convergence is actually in the operator norm (i.e. $D_{k}$ converges to $D$ in $\mathfrak{X}^{\weak}(X)$).

It remains to check that \eqref{basinstreet} holds. From \eqref{equa:abscissaPartialBounds} and the definition of the norm in $\mathfrak{X}^{\weak}(X)$ we have $\sigma^{\weak}_{\mathfrak{X}}(D) \leq \sigma_{\mathfrak{X}^{\weak}(X)}(D)$ for every Dirichlet series with coefficients in $X$. For the reverse inequality we take $D=\sum a_{n} n^{-s}$ and consider $\alpha > \sigma^{\weak}_{\mathfrak{X}}(D)$. Then we define the sequence of bounded linear operators
\[ 
P_{N} : X^{\ast} \longrightarrow \mathfrak{X}, \hspace{3mm} x^\ast \longmapsto \sum_{n=1}^{N} \frac{x^{\ast}(a_{n})}{n^{\alpha}} n^{-s} \,.
\]
For each fixed $x^{\ast} \in X^{\ast}$, the sequence $P_{N}(x^{\ast})$ converges by the choice of $\alpha$, in particular   $\big( P_{N}(x^{\ast}) \big)_{N}$ is bounded. Applying the uniform boundedness principle we deduce that 
\[ \sup_{N}{\Big\| \sum_{n}{\frac{a_{n}}{n^{\alpha}} n^{-s}} \Big\|_{\mathfrak{X}^{\weak}(X)}} = \sup_{N}{\| P_{N}\|} < + \infty. \]
By \eqref{equa:abscissaPartialBounds} we conclude that $\sigma_{\mathfrak{X}^{\weak}(X)}(D) \leq \alpha$.
\end{proof}

We obtain as a consequence the following description of the weak abscissas of uniform and pointwise convergence. This has been independently obtained by J.~Bonet \cite[Proposition~2.1 and Corollary~2.8]{Bo15}.

\begin{Coro} \label{kina}
For every Dirichlet series $D$ in $\mathcal{X}(X)$ we have 
\[
\sigma_{c}(D) = \sigma_{c}^{\weak}(D) \text{ and } \sigma_{u}(D) = \sigma_{u}^{\weak}(D) \,.
\]
\end{Coro}

\begin{proof}
Since abscissas of pointwise and uniform convergence correspond to the convergence in $c(X)$ and $\mathcal{D}_{\infty}(X)$ (see Remark~\ref{stlouis}) it is enough to note that $c^{\weak}(X) = c(X)$ and $\mathcal{D}^{\weak}_{\infty}(X)= \mathcal{D}_{\infty}(X)$ since the supremums commute.
\end{proof}

\begin{Rema}
Using the definition of the weak abscissa of convergence, it follows from the equality of all $\mathcal{H}_{p}(X)$ abscissas for $1 \leq p<\infty$ (see Theorem~\ref{geyrenhoff}) that
\[
\sigma_{\HH_{p}(X)}^{\weak}(D) = \sigma_{\HH_{q}(X)}^{\weak}(D)
\]
for all $1 \leq p < q < \infty$ and every $D$ in $\mathfrak{D}(X)$. That is, all weak abscissas are equal. But, contrary to the abscissas of convergence and of uniform convergence, the weak versions do not coincide with the "classical" one. To see this it suffices to consider the Dirichlet series $\sum_{n}{e_{n}n^{-s}}$ where $\{ e_{n}\colon n \in \N\}$ is the canonical basis of $\ell_{1}$ to have that
\[
\sup_{D}{\sigma_{\HH_{2}(\ell_{1})}^{\weak}(D) - \sigma_{\HH_2(\ell_{1})}(D)} \geq 1/2 \,.
\]
\end{Rema}

\subsection{Abscissa of Unconditional Convergence} \label{sect:unconditionalAbcissa}

Let us recall that a series $\sum x_{n}$ in a Banach space is said to be unconditionally convergent if $\sum \varepsilon_{n} x_{n}$ is convergent for every choice of signs $\varepsilon_{n} = \pm 1$. A classical result of
Riemann shows that unconditional and absolute convergence are the same for finite dimensional spaces, but the Dvoretzki-Rogers Theorem shows that these two concepts are intrinsically different for infinite dimensional 
Banach spaces. It make sense, then, when we deal with Banach-valued Dirichlet series, to consider a new abscissa of unconditional convergence, and to try to distinguish it from the abscissa of absolute convergence.
So, given a Dirichlet series $D= \sum_{n}{a_{n}n^{-s}}$ we define its abscissa of unconditional convergence as
\[ 
\sigma_{\unc}(D) = \inf{\{ \sigma \,\colon\, D  \text{ converges unconditionally on } \mathbb{C}_{\sigma} \}} \,.
\]
In case $X$ is a finite-dimensional Banach space, we have that $\sigma_{a}(D) = \sigma_{\unc}(D)$. So, in what follows we will always assume $X$ to be infinite dimensional and in this case we can only assure that
\begin{equation}  \label{EQUA:threeAbcissas}
\sigma_{u}(D) \leq \sigma_{\unc}(D) \leq \sigma_{a}(D) \,.
\end{equation}
This new abscissa is actually the abscissa of weak absolute convergence that we considered before (and that was also introduced in \cite{Bo15}).
\begin{Lemm}
\label{Lemm:uncondAbcissa}
$\sigma_{\unc}(D) 
= \sigma_{a}^{\weak}(D)$. 
\end{Lemm}
\begin{proof}
First of all, for every $x^{\ast} \in X^{\ast}$ we have that $\sigma_{a}(D_{x^{\ast}}) = \sigma_{\unc}(D_{x^{\ast}}) \leq \sigma_{\unc}(D)$. On the other hand, if $\sum_{n} \big\vert x^{\ast}(a_{n}) \frac{1}{n^{s}} \big\vert <  \infty$ for every $x^{\ast} \in X^{\ast}$, then \cite[Theorem~1.9]{DiJaTo95} gives that $\sum_{n}  a_{n} \frac{1}{n^{s}}$ converges unconditionally.
\end{proof}

Our next step is to see if, as in the case of finite dimensional spaces, these two abscissas of unconditional and absolute convergence are the same or not. Observe that, in view of Lemma~\ref{Lemm:uncondAbcissa} this is basically the
same question as in Proposition~\ref{kina}. As a matter of fact, as next result shows, these two abscissas can be far apart to each other. It has been independently obtained by J.~Bonet in \cite[Theorem~3.6]{Bo15}. The argument we give for the upper bound is essentially the same as there but our approach to the lower bound is different.

\begin{Theo}\label{Prop:stripUnconditional}
For every Banach space $X$ we have
\begin{equation} \label{EQUA:absoluteUncondAbcissa}
\sup \{  \sigma_{a}(D) - \sigma_{\unc}(D)  \,\colon \, D \in \mathfrak{D}(X)  \} = 1 - \frac{1}{\cot(X)}  
\end{equation}
and the supremum is attained.
\end{Theo}
\begin{proof}
Proceeding exactly as in the scalar case we have that $\sigma_{a}(D) - \sigma_{c}(D) \leq 1$ for every $X$-valued Dirichlet series. Then the upper bound is trivial if $\cot{(X)} = \infty$. Now, if $X$ has cotype $2 \leq q < \infty$, we take $\sum_{n} a_{n}n^{-s}$ that converges unconditionally at some $\sigma \in \mathbb{R}$. Then by H\"older's inequality and the definition of cotype (we write $C_{q}(X)$ for the cotype $q$ constant of $X$, $\varepsilon_{n}$ for independent identically distributed Rademacher random variables defined on some probability space $\Omega$ and $q' = q/(q-1)$) we have, for $N<M$
\begin{multline*}
\sum_{n=N}^{M} \frac{\|a_{n}\|}{n^{\sigma + \delta}}  
\leq \bigg( \sum_{n=N}^{M} \Big( \frac{\| a_{n}\|}{n^{\sigma}} \Big)^{q} \bigg)^{1/q} 
\bigg( \sum_{n=N}^{M}{\frac{1}{n^{\delta q'}}} \bigg)^{1/q'}\\  
\leq C_{q}(X) \int_{\Omega} \Big\Vert  \sum_{n=N}^{M} \varepsilon_{n}(\omega)  \frac{ a_{n}}{n^{\sigma}}  \Big\Vert d \omega \bigg( \sum_{n=N}^{M}{\frac{1}{n^{\delta q'}}} \bigg)^{1/q'} \\
\leq C_{q}(X)  \: \sup_{\varepsilon_{n} = \pm 1}  \bigg\| \sum_{n=N}^{M} \varepsilon_{n} \frac{ a_{n}}{n^{\sigma}}   \bigg\| \,\, \bigg( \sum_{n=N}^{M}{\frac{1}{n^{\delta q'}}} \bigg)^{1/q'} \,.
\end{multline*}
Since $\sum_{n} a_{n} / n^{\sigma}$ is unconditionally convergent the supremum is finite and, being $\delta q' > 1$ we conclude that the series $\sum_{n} \| a_{n}\| n^{- (\sigma + \delta)}$ is convergent. This gives
\begin{equation} \label{time}
\sigma_{a}(D) - \sigma_{\unc}(D) \leq  1 - \frac{1}{\cot(X)}
\end{equation}
for every Dirichlet series.\\
For the lower estimate we are going to construct a Dirichlet series for which 
\begin{equation}\label{fade}
\sigma_{a}(D) = 1 - 1/\cot(X) \quad \mbox{ and } \quad \sigma_{\unc}(D) = 0 \,.
\end{equation}
To begin with we fix a strictly increasing sequence of positive numbers $(q_{m})$ converging to $\cot{(X)}$ and a sequence of finite sets of prime numbers $(A_{m})$ satisfying, for each $m \in \N$, that 
\begin{equation} \label{EQUA:indexChoice}
\max{A_{m}} < \min{A_{m+1}} \,, \quad \sum_{\pr \in A_{m}} \frac{1}{\pr^{1 + \frac{q_{m}}{m}}}  < \frac{1}{2^{\frac{q_{m}}{m}}} \: \text{ \hspace{3mm}  and\hspace{3mm} } \sum_{\pr \in A_{m}} \frac{1}{\pr} \geq 1.
\end{equation}
This construction can be easily done by induction using that the series $\sum_{n}{1/\pr_{n}^{\alpha}}$  converges for $\alpha > 1$ and diverges for $\alpha = 1$. \\
By \cite[Theorem~14.5]{DiJaTo95}, given complex numbers  we can find $(x_{\pr})_{\pr \in A_{m}}$ in $X$ such that 
\begin{equation}  \label{EQUA:coefficientChoice}
\frac{1}{2} \sup_{\pr \in A_{m}} \vert a_{\pr} \vert 
\leq \Big\| \sum_{\pr \in A_{m}}{x_{\pr} a_{\pr}} \Big\| 
\leq \Big( \sum_{\pr \in A_{m}} \vert a_{\pr} \vert^{q_{m}} \Big)^{1/q_{m}} \quad \mbox{ for every $(a_{\pr})_{\pr \in A_{m}}$ in $\mathbb{C}$}\,.
\end{equation}
We now define the following Dirichlet series
\begin{equation} \label{eco}
D = \sum_{m} \sum_{\pr \in A_{m}}{\Big( \frac{x_{\pr}}{\pr^{1/q_{m}}} \Big) \frac{1}{\pr^{s}}} \,. 
\end{equation}
Since $1/2 \leq \| x_{\pr}\| \leq 1$ for every $\pr \in \bigcup_{m}{A_{m}}$, the series in \eqref{eco} converges absolutely at a given  $s \in \mathbb{C}$ if and only if
\[ 
\sum_{m} \sum_{\pr \in A_{m}} \frac{1}{\pr^{\Real{(s)} + \frac{1}{q_{m}}}}
\]
converges. If this is the case, by \eqref{EQUA:indexChoice} there must be $m_{0} \in \mathbb{N}$ such that $\Real(s) + 1/q_{m} > 1$ for every $m \geq m_{0}$. But this implies  $\Real(s) \geq 1 - 1/\cot(X)$ and hence $\sigma_{a}(D) \geq 1 - 1/\cot(X)$.\\
On the other hand, for each $\sigma >0$ we can find $m_{0} \in \N$ such that $m \sigma > 1$ whenever $m \geq m_{0}$. Then, for every sequence of signs $\varepsilon_{n} = \pm1$ we have, using \eqref{EQUA:indexChoice} and \eqref{EQUA:coefficientChoice}, that
\[
\Big\| \sum_{\pr \in A_{m}} \frac{\varepsilon_{\pr} x_{\pr}}{\pr^{\sigma + 1/ q_{m}}} \Big\| 
\leq \Big( \sum_{\pr \in A_{m}} \frac{1}{\pr^{1 + q_{m} \sigma}} \Big)^{1/q_{m}} 
\leq \Big( \sum_{\pr \in A_{m}} \frac{1}{\pr^{1 + q_{m}/m}} \Big)^{1/q_{m}} < \frac{1}{2^{m}}
\]
for all $m \geq m_{0}$. This shows that $D$ converges unconditionally at $\sigma$ and therefore $\sigma_{\unc}(D) \leq 0$. Applying now \eqref{time} to the proven estimations we conclude that  \eqref{fade} holds.
\end{proof}

We have separated the abscissas of unconditional and absolute convergence, showing that they may not be equal. The question we address now is the relation between the abscissas of unconditional and uniform 
convergence. They may also be different to each other. Using Corollary~\ref{kina} we can easily pass from scalar to vector-valued Dirichlet series.

\begin{Prop}\label{Prop:unconditionalUniformAbcissa}
Let $X$ be a Banach space. Then
\[
\sup \sigma_{\unc}(D) - \sigma_{u}(D) = \frac{1}{2} \,,
\]
where the supremum is taken over all Dirichlet series with coefficients in $X$. Moreover, the supremum is attained. \\

\end{Prop}

\begin{proof}
Let us point out that if $X = \mathbb{C}$, then both statements hold true \cite[Sections~5 and ~6]{BoHi31} (recall that in this case we have $\sigma_{\unc} = \sigma_{a}$). We take $D \in \mathfrak{D}(X)$ and use this fact and Lemma~\ref{Lemm:uncondAbcissa} to have
\[ 
\sigma_{\unc}(D) = \sup_{x^{\ast} \in B_{X^{\ast}}}  \sigma_{a}(D_{x^{\ast}}) 
\leq \sup_{x^{\ast} \in B_{X^{\ast}}} \sigma_{u}(D_{x^{\ast}}) + \frac{1}{2} 
= \sigma_{u}(D) + \frac{1}{2} \,. 
\]
On the other hand, \cite[Theorem~VII]{BoHi31} provides a Dirichlet series $\sum_{n} \lambda_{n}n^{-s}$ (with $\lambda_{n} \in \mathbb{C}$) such that $\sigma_{\unc} - \sigma_{u} = \frac{1}{2}$. Choosing 
$x \in X$ with $\Vert x \Vert = 1$ and defining $D=\sum_{n} (\lambda_{n}x) n^{-s}$ we have a Dirichlet series with values in $X$ satisfying $\sigma_{\unc}(D) - \sigma_{u}(D) = \frac{1}{2}$. 
\end{proof}

\begin{Rema} \label{xerxes}
Following standard notation, for a given $n \in \mathbb{N}$ the function $\Omega (n)$ counts the number of prime divisors, considered with multiplicity. Then, a Dirchlet series $\sum a_{n} n^{-s}$ is called $m$-homogeneous if 
$a_{n}=0$ for every $n$ with $\Omega (n) \neq m$. The solution to Bohr's problem \eqref{collserola} given in \cite{BoHi31} went through showing that, if the supremum is taken over all $m$-homogeneous Dirichlet series with 
scalar coefficients, then
\[
\sup \{\sigma_{a}(D) - \sigma_{u} (D)\} = \frac{m-1}{2m} \,.
\]
When the study was taken to Dirichlet series in infinite dimensional Banach spaces, it came as a striking result that the dependence on the degree of homogeneity vanishes and
\[
\sup \{ \sigma_{a}(D) - \sigma_{u} (D) \} =  1 - \frac{1}{\cot (X)} \,,
\] 
where the supremum is now taken over all $m$-homogeneous Dirichlet series with coefficients in $X$.\\
If we consider the same question, this time with the abscissa of unconditional convergence, then we can proceed as in the proof of Proposition~\ref{Prop:unconditionalUniformAbcissa} to see that the dependence on the degree again appears and, if we consider the supremum over all $m$-homogeneous Dirichlet series with coefficients in $X$, then
\[
\sup \{ \sigma_{\unc}(D) - \sigma_{u}(D) \} = \frac{m-1}{2m} \,.
\]
\end{Rema}

We finish the section by constructing Dirichlet series for which $\sigma_{a}$ and $\sigma_{u}$ has the maximum difference, while $\sigma_{\unc}$ can be any value allowed by Proposition~\ref{Prop:unconditionalUniformAbcissa}. We need first a lemma.

\begin{Lemm}
\label{Lemm:subspaceCotype}
If $X$ is an infinite dimensional Banach space, then it contains a proper closed subspace $Y \subset X$ with $\cot(Y) = \cot(X)$.
\end{Lemm}

\begin{proof}
Take a norm-attaining element $x_{0}^{\ast} \in X^{\ast}$ with $\Vert x_{0}^{\ast} \Vert_{X^{\ast}} =1$  and consider its kernel $Y:=\ker(x_{0}^{\ast})$. Then $Y$ must be proximinal, so every $x \in X$ can be written as $y + \lambda z$ for a unique pair of elements $y \in Y$ and $\lambda \in \mathbb{K}$. This gives a natural isomorphism $X \cong Y \oplus \mathbb{K}$ which shows that $\cot(X) = \cot(Y \oplus \mathbb{K})$. Moreover $\cot(Y \oplus \mathbb{K}) \leq \cot(Y \oplus Y)$ simply identifying $\mathbb{K} \cong \mathbb{K}y_{0} \subset Y$ for some $y_{0} \in S_{Y}$. Finally $\cot(Y \oplus Y) = \cot(\ell_{2}^{2}(Y)) \leq \cot(Y)$, where the last inequality follows from \cite[Theorem~11.12--(b)]{DiJaTo95}.
\end{proof}

\begin{Coro}
Let $X$ be an infinite dimensional Banach space. For every $0 \leq \alpha \leq 1/2$ there exists a Dirichlet series $D$ such that
\[ \sigma_{u}(D) = 0, \hspace{5mm} \sigma_{\unc}(D) = \alpha, \hspace{5mm} \sigma_{a}(D) = 1 - 1/\cot{(X)} \]
\end{Coro}

\begin{proof}
By Lemma~\ref{Lemm:subspaceCotype}, we can find a proximinal hyperplane $Y$ with $\cot(Y) = \cot(X)$ and $x_{0} \in S_{X}$ with $d(x_{0},Y)=1$. Let $D_{1} = \sum_{n} a_{n} n^{-s}$ be the Dirichlet series with coefficients in $Y$ defined as in \eqref{eco}. By \cite[Theorem~VII]{BoHi31} we can find a scalar-valued Dirichlet series $D_{2} = \sum_{n} \lambda_{n}n^{-s}$ with $\sigma_{a}(D_{2}) = \alpha$ and $\sigma_{u}(D_{2}) = \alpha - 1/2$. Define the Dirichlet series $D$ by
\[ 
 \sum_{n} \frac{a_{n} + \lambda_{n}x_{0}}{n^{s}} \,. 
 \]
The inequality 
\[ 
\sum_{n=N}^{M} \frac{\|a_{n}\|}{n^{\sigma}} 
\leq \sum_{n=N}^{M} \frac{\| a_{n} + \lambda_{n}x_{0} \|}{n^{\sigma}} + \sum_{n=N}^{M} \frac{|\lambda_{n}|}{n^{\sigma}} 
\] 
implies that $1- 1/\cot(X) = \sigma_{a}(D_{1}) \leq \max \{ \sigma_{a}(D), \sigma_{a}(D_{2}) = \alpha \}$, so we deduce that $\sigma_{a}(D) = 1 - 1/\cot(X)$ if $\cot(X)>2$ or $\alpha < 1/2$. Otherwise, $1/\alpha = \cot(X) = 2$ and then $\sigma_{a}(D) = 1/2$ since
\[ 
\sum_{n=N}^{M} \frac{|\lambda_{n}|}{n^{\sigma}} 
\leq \sum_{\substack{N \leq n \leq M \\ \lambda_{n} \neq 0}} \frac{|\lambda_{n}| \cdot \| x_{0} + a_{n}/\lambda_{n} \|}{n^{\sigma}} 
\leq \sum_{n=N}^{M} \frac{\| a_{n} + \lambda_{n} x_{0}\|}{n^{\sigma}} 
\]
implies that $1/2=\sigma_{a}(D_{2}) \leq \sigma_{a}(D) \leq 1/2$. \\
To estimate $\sigma_{\unc}(D)$, notice that for every sequence of signs $\varepsilon_{n} = \pm 1$ for $n \in \N$ and natural numbers $N < M$ we have 
\[
\Big| \sum_{n=N}^{M} \frac{\varepsilon_{n} \lambda_{n}}{n^{\sigma}} \Big| 
\leq \Big\| \sum_{n=N}^{M} \frac{\varepsilon_{n} a_{n}}{n^{\sigma}} 
+ x_{0} \sum_{n=N}^{M} \frac{\lambda_{n} \varepsilon_{n}}{n^{\sigma}} \Big\| 
= \Big\| \sum_{n=N}^{M} \varepsilon_{n} \frac{a_{n} + \lambda_{n}x_{0}}{n^{\sigma}} \Big\| \,. 
\]
With this we deduce that $\sigma_{\unc}(D) \geq \sigma_{\unc}(D_{2}) = \sigma_{a}(D_{2}) = \alpha$. Since $\sigma_{\unc}(D_{1}) = 0 \leq \alpha$ we conclude that $\sigma_{\unc}(D) = \alpha$. Finally, the fact that $\sigma_{u}(D_{1}) = 0$ and $\sigma_{u}(D_{2}) \leq 0$ imply that $\sigma_{u}(D) = 0$.
\end{proof}

\noindent \textbf{Acknowledgement.}
We would like to thank the referee for her/his careful reading and helpful suggestions.

\end{document}